\documentclass[11pt]{amsart}
\usepackage[utf8]{inputenc}
\usepackage{amssymb}
\usepackage{hyperref}
\usepackage{multicol}
\usepackage{colordvi}
\usepackage{graphicx}
\usepackage{pgf,epsfig, epsf}
\numberwithin{equation}{section}
\theoremstyle{plain}
\newtheorem{Theorem}{Theorem}[section]

\newtheorem{Proposition}[Theorem]{Proposition}
\newtheorem{Lemma}[Theorem]{Lemma}
\theoremstyle{definition}

\newtheorem{Example}[Theorem]{Example}
\newtheorem{Definition}[Theorem]{Definition}

\copyrightinfo{}{}
\newcounter{FNC}[page]
\def\fauxfootnote#1{{\addtocounter{FNC}{2}\Magenta{$^\fnsymbol{FNC}$}%
     \let\thefootnote\relax\footnotetext{\Magenta{$^\fnsymbol{FNC}$#1}}}}

\newcommand{\Z}{\mathbb Z}
\newcommand{\R}{\mathbb R}

\newcommand{\conv}{\operatorname{conv}}

\newcommand{\w}{\operatorname{w}}

\newcommand{\ls}{\operatorname{ls_\Delta}}

\newcommand{\lss}{\operatorname{ls_\square}}

\title{Lattice Size in Higher Dimension}

\author{Abdulrahman Alajmi}
\address{Mathematics Department\\
        Public Authority for Applied Education and Training\\
        7WG9+MQ, Ardiya 92400, Kuwait}
\email{aalajmi@kent.edu}

\author{Sayok Chakravarty}
\address{Department of Mathematics, Statistics, and Computer Science\\
University of Illinois Chicago\\
851 S. Morgan Street\\ Chicago, IL 60607, USA}
\email{schakr31@uic.edu}

\author{Zachary Kaplan}
\address{New Jersey  Department of Environmental Protection\\
516 East State Street, Trenton, NJ, 08609, USA
}
\email{Zachary.Kaplan@dep.nj.gov}

\author{Jenya Soprunova}
\address{Department of Mathematical Sciences\\
        Kent State University\\
        800 E. Summit st., Kent, OH 44242, USA}
\email{esopruno@kent.edu}
\urladdr{http://www.math.kent.edu/~soprunova/}

\thanks{Work of Chakravarty, Kaplan, and Soprunova was partially supported by NSF Grant DMS-1653002}


\begin{document}

\begin{abstract} The lattice size of a lattice polytope is a geometric invariant which was formally introduced in the context of simplification of the defining equation  of an algebraic curve, but  appeared implicitly earlier
in geometric combinatorics. Previous work on the lattice size was devoted to studying the lattice size in dimension 2 and 3.
In this paper we establish explicit formulas for the lattice size of a family of lattice simplices in arbitrary dimension. 
\end{abstract}


\maketitle

\section{Introduction} 
This paper is devoted to computing explicitly the lattice size for a family of lattice simplices in $\R^{d+1}$.
We start with recalling some basic definitions related to lattice polytopes.

We say that a point $p\in\R^d$ is a {\it lattice point} if all of its coordinates are integers. A {\it lattice polytope}
$P\subset\R^d$ is the convex hull of finitely many lattice points in $\Z^d$. A {\it lattice segment} is a segment that connects two
lattice points. Such a segment is {\it primitive} if its only lattice points are its endpoints. The {\it lattice length} of a lattice segment
is one less than the number of lattice points it contains (so that a primitive segment has lattice length one). A lattice polytope is {\it empty} if its only lattice points are its vertices.

We say that matrix $A$ of size $d$ with integer entries is {\it unimodular} if $\det(A)=\pm 1$. The set of such matrices is denoted by ${\rm GL}(d,\Z)$.
We say that a map $L\colon\R^d\to\R^d$ is an {\it affine unimodular map} if it is a composition of multiplication by a unimodular matrix and a translation by an integer vector. Such maps preserve the integer lattice $\Z^d\subset\R^d$. 
We say that two lattice polytopes in $\R^d$ are {\it lattice-equivalent} if one is the image of another under an affine unimodular map.

Let $h$ be an integer vector in $\R^d$. For a lattice polytope $P\subset\R^d$, we define the {\it lattice width of $P$ in the direction of $h$} by
$$\w_h(P)=\max\limits_{x\in P}\langle h, x\rangle - \min\limits_{x\in P}\langle h, x\rangle,
$$
where $\langle h, x\rangle $ is the standard inner product in $\R^d$. Then the {\it lattice width} $\w(P)$ of $P$ is the minimum of $\w_h(P)$ over nonzero integer vectors 
$h\in\Z^d$.

 \begin{figure}[h]
\begin{center}
\includegraphics[scale=.8]{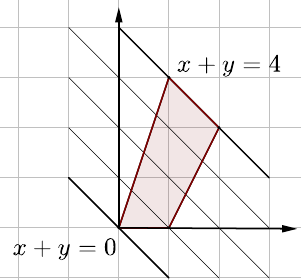} 
\caption{Lattice polygon $P$ with $\operatorname{w}_{(1,1)}(P)=4$.}
\label{F:width-def} 
\end{center}
\end{figure}

In Figure~\ref{F:width-def}, we illustrate the geometric meaning of this definition. Polygon $P$ in the diagram is squeezed between the lines $x+y=0$ and $x+y=4$, so its width in the direction $(1,1)$ is 4.
We also have $\w_{(1,0)}(P)=2$. Since $P$ has interior lattice points, we conclude that $\w(P)=2$.

The lattice size of a lattice polytope is an important geometric invariant of a lattice polytope that was formally introduced in~\cite{CasCools}, but appeared
implicitly earlier in~\cite{Arnold, BarPach, BrownKasp,LagZieg, Schicho}. It was further studied in~\cite{AlajmiSopr, HarSopr, HarSoprTier,Sopr}.

We next reproduce the definition of the lattice size from~\cite{CasCools}.
Let $0\in\R^d$ be the origin and let $(e_1,\dots ,e_d)$ be the standard basis of $\R^d$. The {\it standard simplex} $\Delta\subset\R^d$ is defined by
$\Delta=\conv\{0,e_1,\dots,e_d\}$, where $\conv$ denotes the convex hull operator.

\begin{Definition} Let $P\subset\R^d$ be  a lattice polytope. The lattice size $\ls(P)$ of $P$ with respect to the standard simplex $\Delta$ is the smallest $l$ such that
$L(P)$ is contained in the $l$-dilate of $\Delta$ for some affine unimodular map $L\colon\R^d\to\R^d$.
\end{Definition}
Equivalently, if we let
\begin{align}\label{a:defl_1}
l_1(P)=\max\limits_{(x_1,\dots,x_d)\in P} (x_1+\cdots+x_d)-\min\limits_{(x_1,\dots,x_d)\in P} x_1-\cdots-\min\limits_{(x_1,\dots,x_d)\in P} x_d,
\end{align}
then $\ls(P)$ is the minimum of $l_1(L(P))$ over affine unimodular maps $L\colon\R^d\to\R^d$.

If in the above definition the standard simplex $\Delta$ is replaced with the unit cube $\square=[0,1]^d$, we obtain the definition of the lattice size $\lss(P)$
with respect to the unit cube. Note that the lattice width $\w(P)$ can be viewed as the lattice size with respect to the strip $\R^{d-1}\times [0,1]$.

 \begin{figure}[h]
\begin{center}
\includegraphics[scale=.7]{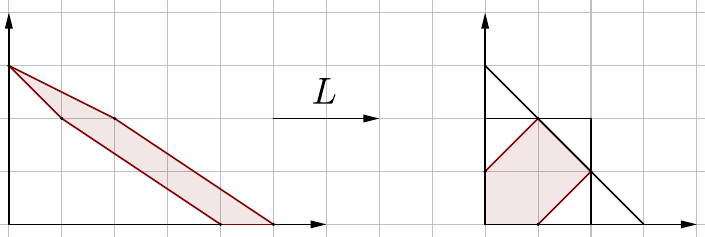} 
\caption{}
\label{F:ls_example} 
\end{center}
\end{figure}

 \begin{Example} 
  Let $P$ be the polygon with  vertices $(4,0), (5,0), (2,2), (0,3)$, and $(1,2)$, as drawn in Figure~\ref{F:ls_example}. Define
 $L(x,y)=\begin{bmatrix}1&1\\-1&-2\end{bmatrix}\cdot \begin{bmatrix}x\\ y\end{bmatrix}+\begin{bmatrix}-3\\ 6\end{bmatrix}$. We get
 $$L(P)=\conv\{(1,2), (2,1), (1,0), (0,0), (0,1)\},$$ so  $L(P)\subset 2\square$ and  $L(P)\subset 3\Delta$. Note that $P$
 has an interior lattice point, while $\square$ and $2\Delta$ do not, so  it is impossible to unimodularly map $P$ inside $\square$ or $2\Delta$.
 We conclude that $\lss(P)=2$ and $\ls(P)=3$.
 \end{Example}

It was shown in~\cite{HarSopr, HarSoprTier} that in dimension 2 both $\ls(P)$ and $\lss(P)$ can be computed using basis reduction (see~\cite{HarSopr, HarSoprTier} for definitions and details). It is further explained in~\cite{HarSopr} that basis reduction also computes the lattice size $\lss(P)$ in dimension 3.
This leads to fast algorithms for computing the lattice size in these cases. A counterexample in~\cite{HarSopr} demonstrates that a reduced basis does not necessarily compute $\ls(P)$ in dimension 3.

A well-known classification result by White~\cite{White} asserts that up to lattice equivalence empty lattice tetrahedra  in $\R^3$ are of the form
$$T_{pq}=\begin{bmatrix}
1&0&0&p\\
0&1&0&q\\
0&0&1&1
\end{bmatrix},
$$ 
where $p$ and $q$ are nonnegative relatively prime integers.  (Note that in this notation $T_{pq}$ is the convex hull of the column vectors of the matrix.)
It was further shown in~\cite{Scarf} that any empty lattice polytope in $\R^3$ has lattice width one.

While it is not true that a reduced basis computes $\ls(P)$ for $P\subset\R^3$, it was shown in~\cite{AlajmiSopr} that this is the case for 3-dimensional empty lattice polytopes. A counterexample was provided in~\cite{AlajmiSopr}  demonstrating that the conclusion does not generalize to all lattice width one polytopes 
$P\subset\R^3$.

All the results discussed above concern the lattice size of lattice polytopes in $\R^2$ and $\R^3$. In this paper, we consider a family of lattice simplices $P$ in 
$\R^{d+1}$ for arbitrary $d$ and explicitly compute both $\ls(P)$ and $\lss(P)$ under some assumptions on the parameters of the family. We work with simplices of the form 
$$T_{p_1\dots p_d}=\begin{bmatrix}
1&0&\dots&0&0&p_1\\
0&1&\dots&0&0&p_2\\
\vdots&  &\ddots&\vdots&\vdots&\vdots\\
0&0&\dots&1&0&p_d\\
0&0&\dots&0&1&1
\end{bmatrix},
$$
where $p_1,\dots, p_d$ are nonnegative integers. These simplices are a natural $(d+1)$-dimensional generalization of the tetrahedra $T_{pq}\subset\R^3$. Each such $T_{p_1\dots p_d}$ has lattice width one. Also, it is empty if and only if
$\gcd(p_1,\dots,p_d)=1$, although we will not be making this assumption. Note that starting with dimension 4  it is no longer true that every empty lattice simplex has lattice width one, see~\cite{HaaseZieg}.

Our main results are formulated in Theorems~\ref{T:Main1} and~\ref{T:Main2}, where we provide explicit formulas for both $\ls(T_{p_1\dots p_d})$ and  $\lss(T_{p_1\dots p_d})$ in terms of $p_1,\dots, p_d$ under some restrictions on these parameters. Our methods are elementary and different from the ones used in earlier work in dimension 2 and 3.

\section{First Lemmas}
Here we provide the $d$-dimensional version of the introductory statements about the lattice size, which were formulated in~\cite{AlajmiSopr} for the case $d=3$.

Let $P\subset\R^d$ be a lattice polytope and $A\in{\rm GL}(d,\Z)$. Denote the rows of $A$ by $h_1,\dots, h_d$. Recall the definition of $l_1(P)$ in~(\ref{a:defl_1}).

\begin{Lemma}\label{L:IntroLemma1}
\begin{itemize} 
\item[(1)] For any $h\in \Z^d$ we have $\w_{h}(AP)=\w_{A^{T}h}(P)$.
\item[(2)] For $i=1,\dots, d$ we have $\w_{e_i}(AP)=\w_{h_i}(P)$.
\item[(3)] Let $(e_1,\dots,e_d)$ be the standard basis of $\R^d$. Then $\w_{e_i}(P)\leq l_1(P)$ for $i=1,\dots, d$. 
\item[(4)] For $i=1,\dots, d$ we have $\w_{h_i}(P)\leq l_1(AP)$.
\item[(5)] Let $e\in\Z^d$ be a vector whose entries lie  in $\{0,1\}$. Then $\w_{e}(P)\leq l_1(P)$.
\item[(6)] Let $h$ be the sum of any non-empty collection of rows of $A$. Then $\w_{h}(P)\leq l_1(AP)$.
\end{itemize}
\end{Lemma}

\begin{proof}
For (1) we have
$$\w_h(AP)=\max\limits_{x\in P}\langle h,Ax\rangle- \min\limits_{x\in P}\langle h,Ax\rangle=\max\limits_{x\in P}\langle A^Th,x\rangle- \min\limits_{x\in P}\langle A^Th,x\rangle=\w_{A^Th}(P),
$$
and (2) is a particular case of (1).  

To check (3), denote $l_1:=l_1(P)$. Then $P\subset l_1\Delta$ and hence 
$$\w_{e_i}(P)\leq \w_{e_i}(l_1\Delta)=l_1=l_1(P).$$
Then (4) follows as $\w_{h_i}(P)=\w_{e_i}(AP)\leq l_1(AP).$

Next we check (5), which is similar to (3): $\w_{e}(P)\leq \w_{e}(l_1\Delta)=l_1=l_1(P).$

For (6) let $e$ be the sum of the corresponding standard basis vectors. Then by (1) and~(5)
$$\w_{h}(P)=\w_{e}(AP)\leq l_1(AP).$$
\end{proof}

\begin{Lemma}\label{L:IntroLemma2} We have
\begin{itemize} 
\item[(1)] 
$l_1(AP)=\max_{x\in P}\langle h_1+\cdots+h_d,x\rangle-\min_{x\in P}\langle h_1,x\rangle-\cdots-\min_{x\in P}\langle h_d,x\rangle$;
\item[(2)] $l_1(AP)$ does not depend on the order of rows in $A$;
\item[(3)] $
l_1(AP)=l_1(BP)$, where $B=\left(\begin{matrix}h_1\\ \vdots\\h_{d-1}\\
-\sum_{i=1}^d h_i
\end{matrix}\right).
$
\end{itemize}
\end{Lemma}

\begin{proof} (1) and (2) are clear.
Let's check (3) using (1):
\begin{align*}
l_1\left(BP\right)&=\max\limits_{x\in P}\langle-h_d,x\rangle
-\sum_{i=1}^{d-1}\min\limits_{x\in P}\langle h_i,x\rangle
-\min\limits_{x\in P}\langle(-h_1-\cdots-h_d),x\rangle
=l_1(AP).
\end{align*}

\end{proof}

\section{Lattice size computation}
Recall that 
$T_{p_1\dots p_d}=\conv\{e_1,\dots, e_{d+1}, (p_1,\dots, p_d,1)\}\subset\R^{d+1}, 
$
where $p_i\in\Z_{\geq 0}$. Define $\alpha=p_1+\cdots+p_{d-1}$. Assume that $p_d\geq 2$ and let $k=\lfloor\frac{p_d-2}{\alpha+1}\rfloor$.

\begin{Proposition}\label{P: upper_bound} With $p_i$, $\alpha$, and $k$  defined as above, suppose that $p_d\geq \alpha^2-\alpha$. Then
$\ls(T_{p_1 \dots p_d})\leq k+3$.
\end{Proposition}

\begin{proof} 
Consider unimodular matrix $A$ of size $d+1$ defined by
$$A=\left(\begin{matrix}
1&0&\dots&0&0&0\\
0&1&\dots&0&0&0\\
\vdots&\vdots &\ddots&\vdots&\vdots&\vdots\\
0&0&\dots&1&0&0\\
0&0&\dots&0&0&1\\
k+1&k+1&\dots&k+1&-1&p_d-\alpha(k+1)-1
\end{matrix}\right).
$$ 
Then 
$$AT_{p_1 \dots p_d}=\begin{bmatrix}
1&0&\dots&0&0&p_1\\
0&1&\dots&0&0&p_2\\
\vdots&\vdots &\ddots&\vdots&\vdots&\vdots\\
0&0&\dots&1&0&p_{d-1}\\
0&0&\dots&0&1&1\\
k+1&k+1&\dots&k+1&p_d-\alpha(k+1)-1&-1
\end{bmatrix},
$$ 
and hence 
$$l_1(AT_{p_1\dots p_d})=\max\{k+2,p_d-\alpha(k+1),\alpha\}-\min\{-1,p_d-\alpha(k+1)-1\}.$$
Since $k=\lfloor\frac{p_d-2}{\alpha+1}\rfloor$, we have $k+1>\frac{p_d-2}{\alpha+1}$, which implies $p_d-\alpha(k+1)< k+3$ and, since $p_d, k$ and $\alpha$ are integers, we 
conclude $p_d-\alpha(k+1)\leq k+2$. We also have
$$k+2>\frac{p_d-2}{\alpha+1}+1=\frac{p_d+\alpha-1}{\alpha+1}\geq \frac{\alpha^2-\alpha+\alpha-1}{\alpha+1}=\alpha-1,
$$
where  we used the assumption $p_d\geq \alpha^2-\alpha$. We have checked that $\max\{k+2,p_d-\alpha(k+1),\alpha\}=k+2$.

Next note that since $\frac{p_d-2}{\alpha+1}\geq k$ we get
$$p_d-1-\alpha(k+1)\geq p_d-1-\alpha\left(\frac{p_d-2}{\alpha+1}+1\right)=\frac{-\alpha^2-1+p_d}{\alpha+1}\geq\frac{-\alpha^2-1+\alpha^2-\alpha}{\alpha+1}=-1.
$$
Hence $\min\{-1,p_d-\alpha(k+1)-1\}=-1$ and $l_1(AT_{p_1 \dots p_d})=k+3$, which implies $\ls(T_{p_1 \dots p_d})\leq k+3$.
\end{proof}

Our next goal is to show that under our assumptions we have $\ls(T_{p_1\dots p_d})=k+3$. For this, we first prove two lemmas.

\begin{Lemma}\label{L:Lemma1} Let $h=(a_1,\dots, a_{d+1})\in\R^{d+1}$ be a primitive vector with $\w_{h}(T_{p_1\dots p_d})\leq k+2$.
Then $a_d\in\{0,1,-1\}$.
\end{Lemma}

\begin{proof} We can assume $a_d\geq 0$. Suppose that $a_d\geq 2$. We have
$$\max\{a_1,\dots,a_{d+1},a_1p_1+\cdots +a_dp_d+a_{d+1}\}-\min\{a_1,\dots,a_{d+1},a_1p_1+\cdots +a_dp_d+a_{d+1}\}\leq k+2,
$$
which implies $a_1p_1+\cdots +a_dp_d\le k+2$
and $a_d-a_i\leq k+2$ for $i=1,\dots,d-1$.
Since $a_d\geq 2$, the first inequality  implies
$$a_1p_1+\cdots +a_{d-1}p_{d-1}\leq k+2-2p_d.
$$
The second inequality implies $a_i\geq -k$ for $i=1,\dots,d-1$ and  hence
$$a_1p_1+\cdots +a_{d-1}p_{d-1}\geq -k(p_1+\cdots+p_{d-1})=-k\alpha.
$$
Combining this with what we got from the first inequality, we get
$k+2-2p_d\geq -k\alpha.$
Hence, using the definition of $k$, we get
$$2p_d\leq k+2+k\alpha=k(\alpha+1)+2\leq p_d,
$$
  which contradicts the assumption $p_d>0$.
\end{proof}

\begin{Lemma}\label{L:Lemma2} Let $h=(a_1,\dots, a_{d+1})\in\R^{d+1}$ be a primitive vector with $a_d=\pm 1$ and  $\w_{h}(T_{p_1\dots p_d})\leq  k+2$.
Then $\w_{h}(T_{p_1\dots p_d})= k+2$.
\end{Lemma}

\begin{proof}
We can assume that $a_d=1$. Suppose that 
\begin{align*}
&\max\{a_1,\dots,a_{d-1},1, a_{d+1},a_1p_1+\cdots +a_{d-1}p_{d-1}+p_d+a_{d+1}\}\\
&-\min\{a_1,\dots,a_{d-1},1,a_{d+1},a_1p_1+\cdots +a_{d-1}p_{d-1}+p_d+a_{d+1}\}\leq k+1.
\end{align*}
This implies $a_1p_1+\cdots+a_{d-1}p_{d-1}+p_d\leq k+1$ and $1-a_i\leq k+1$ for $i=1,\dots, d-1$.
Hence for such $i$ we have $a_i\geq -k$ and
$$-k\alpha+p_d\leq a_1p_1+\cdots+a_{d-1}p_{d-1}+p_d\leq k+1,
$$
which implies $-k\alpha+p_d\leq k+1$.
Hence  $p_d\leq k\alpha+k+1=k(\alpha+1)+1\leq p_d-1$,
and this  is impossible.
\end{proof}

\begin{Theorem}\label{T:Main1} 
Let $\alpha=p_1+\cdots+p_{d-1}$, where all $p_i$ are positive and $p_d\geq 2$. Define $k=\lfloor\frac{p_d-2}{\alpha+1}\rfloor$.
Suppose that $p_d\geq \alpha^2-\alpha$. Then $\ls(T_{p_1\dots p_d})=k+3$.
\end{Theorem}

\begin{proof}
Suppose that there exists a unimodular map $L$ that maps $T_{p_1\dots p_d}$ inside $(k+2)\Delta$ and let $A$ be the corresponding 
unimodular matrix. Then by Lemma~\ref{L:IntroLemma1} for each of its rows  $h$ we have $\w_{h}(T_{p_1\dots p_d})\leq k+2$. We also have the same 
inequality for the sum of any nonempty collection of rows of $A$. By Lemma~\ref{L:Lemma1} each of the entries in the $d$th column of $A$
is 0, 1, or $-1$, and the same applies to the sum of any collection of entries in the  $d$th column of $A$. Hence, up to permutation of rows, the $d$th column
of $A$ is $(0,0,\dots, 0,\pm 1)^T$ or $(0,0,\dots, 0,1, -1)^T$ and by Lemma~\ref{L:IntroLemma2} we can assume that it is the former. Then by Lemma~\ref{L:Lemma2} we have $\w_{e_{d+1}}(L(T_{p_1\dots p_d}))=k+2$.
Since $L(T_{p_1\dots p_d})\subset (k+2)\Delta$ we conclude that 
$$(0,\dots, k+2)\in L(T_{p_1\dots p_d}).$$ 
Similarly, we have $\w_{e_1+\cdots+e_{d+1}}(L(T_{p_1\dots p_d}))=k+2$ and, together 
with $L(T_{p_1\dots p_d})\subset (k+2)\Delta$, this implies that $L(T_{p_1\dots p_d})$ contains the origin. Hence $L(T_{p_1\dots p_d})$ and, therefore, $T_{p_1\dots p_d}$ contains an edge of lattice length 
$k+2$. All the edges of $T_{p_1\dots p_d}$ are primitive, except, possibly, for the one connecting points $(0,\dots,1)$ and $(p_1,\dots,p_d,1)$, whose lattice length is $\gcd(p_1,\dots,p_d)$.
Hence we conclude that $\gcd(p_1,\dots,p_d)=k+2$.   Using the assumption $p_d\geq \alpha^2-\alpha$ we get
$$k+2=\gcd(p_1,\dots, p_d)\leq p_1+\dots+p_{d-1}=\alpha\leq \frac{p_d-2}{\alpha+1}+2<k+3.
$$

Note that since all the $p_i$ are positive, we have $\gcd(p_1,\dots, p_d)< p_1+\dots+p_{d-1}$ unless $d=2$ and $p_2$ is a multiple of of $p_1$.
When the inequality is strict we arrive  at a contradiction since then integer $p_1+\dots+p_{d-1}$ is strictly between the consecutive integers $k+2$ and $k+3$.

It remains to consider the case when $d=2$, $p_2$ is a multiple of $p_1$, and $k+2=\alpha=p_1$.  
We have $\lfloor\frac{p_2-2}{p_1+1}\rfloor=k=p_1-2$ and hence
$$p_2-2=(p_1+1)(p_1-2)+r,
$$
where $0\leq r\leq p_1$. We get  $p_2=p_1^2-p_1+r$ and, since $p_2$ is a multiple of $p_1$, there are two options: $p_2=p_1^2$ and $p_2=p_1^2-p_1$.

Suppose first $p_2=p_1^2$ and let $h=(a_1, a_2,a_3)$ be a direction with $\w_h(T_{p_1p_2})\leq k+2=p_1$.
By Lemma~\ref{L:Lemma1} we have $a_2=0,\pm 1$.  For  $a_2=-1$ we get
$$\max\{a_1,-1,a_3,a_1p_1-p_1^2+a_3\}-\min\{a_1,-1,a_3,a_1p_1-p_1^2+a_3\}\le p_1,
$$
which implies $a_1+1\leq p_1$ and $-a_1p_1+p_1^2\leq p_1$, so we conclude that $a_1=p_1-1$. Plugging in this value for $a_1$ we get
$$\max\{p_1-1,-1,a_3,a_3-p_1\}-\min\{p_1-1,-1,a_3,a_3-p_1\}\le p_1,
$$
which implies $a_3+1\leq p_1$ and $p_1-1-(a_3-p_1)\leq p_1$, so $a_3=p_1-1$. We conclude that for $a_2=-1$ the only $h$ with $\w_h(T_{p_1p_2})\leq p_1$ is 
$h=(p_1-1,-1,p_1-1)$ and for such $h$ we get $\w_h(T_{p_1p_2})=k+2$. Similarly, for $a_2=1$ such direction is $h=(1-p_1,1,1-p_1)$.

Hence  in this case we can only use as rows of $A$ vectors $\pm (p_1-1,-1,p_1-1)$ and vectors whose second component is 0.
Further, we can assume that the second column of $A$ is $(0,0,\pm 1)^T$, and the third row is $\pm (p_1-1,-1,p_1-1)$.
Then the sum of the third row with any of the first two rows will also have to be of the same form as the third row, which would imply $\det A=0$.

We next consider the last case $p_2=p_1^2-p_1$. Recall that we have $k+2=p_1$. By Lemma~\ref{L:Lemma1} and Lemma~\ref{L:Lemma2} if $\w_h(T_{p_1p_2})\leq p_1$ for 
$h=(a_1,a_2,a_3)$ then $a_2=0,\pm 1$ and if $a_2=\pm 1$ we have $\w_h(T_{p_1p_2})= p_1$. Let's further investigate the case $a_2=0$. We have 
$$\max\{a_1,0,a_3,a_1p_1+a_3\}-\min\{a_1,0,a_3,a_1p_1+a_3\}\le p_1,
$$
where we can assume $a_1\geq 0$. Hence $a_1p_1\leq p_1$, which implies $a_1=0$ or $1$ and in the latter case the width is $p_1$. We conclude that the only direction $h$ with $\w_h(T_{p_1p_2})< p_1$ is 
$(0,0,\pm 1)$.

As before, we can assume that the second column in matrix $A$ is $(0,0,1)^T$. Denote the rows of $A$ by $h_1,h_2$, and $h_3$.
Then out of the widths of $T_{p_1p_2}$ in the direction of  $h_1$ and $h_2$ only one can be strictly less than $p_1$.
We can assume that this happens in the direction of $h_1$. We also know that the width of $T_{p_1p_2}$ in the direction of $h_3$ and $h_1+h_2+h_3$
is $p_1$. We can now conclude that $L(T_{p_1p_2})$ contains the triangle 
$$\conv\{(0,0,0), (0,p_1,0), (0,0,p_1)\}.
$$
Note that we assumed that $p_d\geq 2$, which implies $p_1^2-p_1=p_2\geq 2$ and hence $p_1\geq 2$. Hence our conclusion implies  that  $T_{p_1p_2}$ contains a lattice triangle with three 
non-primitive sides, and this contradiction completes the argument.
\end{proof}

We next use the above work  to compute $\lss(T_{p_1,\dots, p_d})$.

\begin{Theorem}\label{T:Main2}  Let $\alpha=p_1+\cdots+p_{d-1}$, where all $p_i$ are positive and $p_d\geq 2$. Define $k=\lfloor\frac{p_d-2}{\alpha+1}\rfloor$. Suppose that $p_d\geq \alpha^2-\alpha$. Then $\lss(T_{p_1\dots p_d})=k+2$.
\end{Theorem}
\begin{proof}
Let $A$ be the matrix from Proposition~\ref{P: upper_bound}. Then $\w_{e_{d}}(AT_{p_1\dots p_d})=1$, $\w_{e_i}(AT_{p_1\dots p_d})=p_i$ for $i=1,\dots, d-1$, 
and 
$$\w_{e_{d+1}}(AT_{p_1\dots p_d})=\max\{k+1,p_d-\alpha(k+1)-1\}-\min\{-1,p_d-\alpha(k+1)-1\}=k+2,
$$
as shown in the proof of Proposition~\ref{P: upper_bound}. It is also checked there that $\alpha\leq k+2$, and hence after a translation by the vector $v=(0,\dots,0,1)^T$ we get
$AT_{p_1\dots p_d}+v\subset  [0,k+2]^{d+1}$, which implies that $\lss(P)\leq k+2$. 

Suppose next that there exists a unimodular map $L\colon\R^d\to\R^d$ such that $L(T_{p_1\dots p_d})\subset  [0,k+1]^{d+1}$. Then the width of $T_{p_1\dots p_d}$ in the direction of each of the rows of the corresponding matrix $A$ is at most $k+1$, but by Lemmas~\ref{L:Lemma1} and ~\ref{L:Lemma2} this implies that the $d$th entry of each row of $A$ is zero, so $\det A=0$.

\end{proof}

Note that we have checked in Theorems~\ref{T:Main1} and~\ref{T:Main2} that there exists  matrix $A$ that computes both $\ls(T_{p_1\dots p_d})$ and $\lss(T_{p_1\dots p_d})$. While it was shown in~\cite{HarSopr, HarSoprTier} that this is always the case when $P\subset\R^2$, a counterexample for  $P\subset\R^{3}$ was provided in~\cite{HarSopr}.

\subsection*{Acknowledgments} 
We are grateful to the Kent State REU program for the hospitality. We would also like to thank the anonymous referee for useful suggestions.

\end{document}